\documentclass[10pt]{amsart}

\usepackage{amsmath} \usepackage{amssymb, amscd,cancel, graphicx,soul,stmaryrd}
	
\usepackage{pstricks,pinlabel,mathtools}

\usepackage{enumerate}

\usepackage{tikz-cd}

\usepackage{mathdots}
\newtheorem{thm}{Theorem}[section]

\newtheorem{lem}[thm]{Lemma}
\newtheorem{prop}[thm]{Proposition}

\newtheorem*{ques*}{Question}

\newtheorem*{example*}{Example}
\newtheorem*{porism*}{Porism}
\newtheorem*{scholium*}{Scholium}

\newtheorem*{thm*}{Theorem}
\newtheorem*{defin*}{Definition}
\newtheorem*{lem*}{Lemma}

\def\cN{{\mathcal N}}

\def\C{{\mathbb C}}

\def\R{{\mathbb R}}
\def\bR{{\mathbb R}}

\def\bZ{{\mathbb Z}}

\def\Sym{{\mathrm{Sym}}}

\def\del{{\partial}}

\def\im{{\text{im}}}

\begin{document}
\thispagestyle{empty}
\title[The rectangular peg problem]{The rectangular peg problem}
\author{Joshua Evan Greene} 
\address{Department of Mathematics, Boston College, USA}
\email{joshua.greene@bc.edu}
\urladdr{https://sites.google.com/bc.edu/joshua-e-greene}
\author{Andrew Lobb} 
\address{Mathematical Sciences,
	Durham University,
	UK}
\email{andrew.lobb@durham.ac.uk}
\urladdr{http://www.maths.dur.ac.uk/users/andrew.lobb/}
\thanks{AL thanks the Okinawa Institute of Science and Technology for hosting him as Excellence Chair while this work was completed.  JEG was supported on NSF CAREER Award DMS-1455132 and a Simons Fellowship.}

\begin{abstract}
For every smooth Jordan curve $\gamma$ and rectangle $R$ in the Euclidean plane, we show that there exists a rectangle similar to $R$ whose vertices lie on $\gamma$.
The proof relies on Shevchishin's theorem that the Klein bottle does not admit a smooth Lagrangian embedding in $\C^2$.
\end{abstract}

\maketitle

The result of this paper is the solution of the Rectangular Peg Problem for smooth Jordan curves:
\begin{thm*}
	\label{thm:main_theorem}
For every smooth Jordan curve $\gamma$ and rectangle $R$ in the Euclidean plane, there exists a rectangle similar to $R$ whose vertices lie on $\gamma$.
\end{thm*}

\begin{proof}
We consider the Euclidean plane $\C$ with complex coordinate $z = x + i \cdot y$ and take another copy with complex coordinate $w = r \cdot e^{i \theta}$.
In these coordinates, the standard symplectic structure on $\C^2$ is given by $\omega = dx \wedge dy + r \cdot dr \wedge d \theta$.
Consider the maps $l, g : \C^2 \to \C^2$ defined in a mix of complex and polar coordinates by
\[
l \colon (z,w) \mapsto \left( \frac{z+w}{2},\frac{z-w}{2}\right) \quad \textup{and} \quad g \colon (z,r,\theta) \mapsto (z,r/\sqrt{2},2 \theta) {\rm .}
\]
The map $l$ is a diffeomorphism and satisfies $l^*(\omega) = \omega/2$.  Away from $\C \times \{ 0 \}$, the map $g$ is smooth and satisfies $g^* (\omega) = \omega$.
The Jordan curve $\gamma$ is Lagrangian in $\C$, so both the product $\gamma \times \gamma$ and its image $L = l(\gamma \times \gamma)$ are smooth, Lagrangian tori in $\C^2$,
noting that Lagrangians with respect to $\omega$ coincide with those with respect to $\omega/2$.
For any $\phi \in \R$, the map
\[
R_\phi : \C^2 \to \C^2 \colon (z,r,\theta) \mapsto (z,r,\theta+\phi)
\]
is a symplectomorphism.  Fixing a choice $0 < \phi \leq \pi/2$, $L_\phi = R_\phi(L)$ is another smooth, Lagrangian torus.
By construction, $g \circ l (z,w) = g \circ l (z',w')$ if and only if $\{z,w\} = \{z',w'\}$.
It follows that $M = g(L)$ and $M_\phi = g(L_\phi)$ are both homeomorphic to a M\"obius band $\Sym^2(\gamma)$ and are smooth and Lagrangian away from $\C \times \{0\}$.

The map $R_\pi$ preserves each of $L$ and $L_\phi$, and it fixes $\gamma \times \{0\}$, where these two tori intersect {\em cleanly}: $T_p(\gamma \times \{0\}) = T_pL \cap T_p L_\phi$ at each point $p \in \gamma \times \{ 0 \}$.
We perform a Lagrangian smoothing of $L \cup L_\phi$ along $\gamma \times \{ 0 \}$ according to Proposition \ref{prop:tori_can_be_smoothed_nicely} below.
The result is a smoothly immersed Lagrangian torus in $\C \times \C$ that coincides with $L \cup L_\phi$ away from a neighborhood of $\gamma \times \{ 0 \}$, is disjoint from $\C \times \{ 0 \}$, and on which $R_\pi$ acts as a fixed-point free involution.
Its image under $g$ is therefore a smoothly immersed, Lagrangian Klein bottle $K$ which coincides with $M \cup M_\phi$ outside of a neighborhood of $\gamma \times \{ 0 \}$ and is embedded within this neighborhood.
Shevchishin has shown that there is no smoothly embedded, Lagrangian Klein bottle in $\C^2$~\cite{shevchishin2009}.
Therefore, $M$ and $M_\phi$ must intersect at a point away from $\gamma\times\{ 0 \}$, so $L$ and $L_\phi$ do as well, say at the point $(z, re^{i(\theta + \phi)})$.
It follows that the four points $z \pm re^{i \theta}$ and $z \pm re^{i(\theta + \phi)}$ all lie on the Jordan curve $\gamma$.
These points form the vertices of a rectangle whose diagonals meet at an angle of $\phi$.
As $\phi \in (0,\pi/2]$ was arbitrary, the proof is complete.
\end{proof}

The proof establishes somewhat more:
\begin{porism*}
For every smooth Jordan curve $\gamma$ and smooth map $\phi \colon [0, \infty) \rightarrow (0,\pi)$, there exists $r>0$ such that $\gamma$ contains the vertices of a rectangle of diameter $r$ whose diagonals meet at angle $\phi(r)$.
\qed
\end{porism*}

\noindent
We simply note that the map
\[ S_\phi \colon \C^2 \rightarrow \C^2 \colon (z,r,\theta) \mapsto (z, r, \theta + \phi(2r)) \]
is a symplectomorphism, so $L_\phi = S_\phi(L)$ is a Lagrangian torus, invariant under $R_\pi$ and meeting $L$ cleanly along $\gamma \times \{ 0 \}$.
The main result covers the case of a constant function $\phi$.

\section{Lagrangian smoothing.}
We now turn to the smoothing used in the proof of the theorem:

\begin{prop}
	\label{prop:tori_can_be_smoothed_nicely}
	One may remove a neighborhood of $\gamma \times \{ 0 \}$ in $L  \cup L_\phi$ and replace it with two disjoint Lagrangian annuli.
	The surgery may be performed so as to result in a smoothly immersed Lagrangian torus $T$ such that $R_\pi(T) = T$ and $T$ is disjoint from $\C \times \{ 0 \}$.
\end{prop}
\noindent Here $L_\phi$ may denote either $R_\phi(L)$ or $S_\phi(L)$ from the previous section.
{\em A posteriori} we obtain a Lagrangian smoothing of $M \cup M_\phi$ nearby the common boundary $\del M = \del M_\phi = \gamma \times \{ 0 \} \subset \C \times \{ 0 \}$, but we found it more direct to work rather with $L \cup L_\phi$, due to the non-smoothness of $g$ at $\C \times \{ 0 \}$.

Proposition~\ref{prop:tori_can_be_smoothed_nicely} will not come as a surprise to symplectic geometers, although we could not locate the desired result in the literature.
It can be phrased as a consequence of a simple case of the equivariant Darboux-Weinstein theorem in the presence of a compatible clean intersection of Lagrangians.
We shall prove the proposition by establishing a linear local model for $L \cup L_\phi$ near $\gamma \times \{0\}$. 
The local model is the 4-manifold $X = S^1 \times (-\epsilon, \epsilon) \times \bR \times \bR$ with coordinates $(\theta, s, t_1, t_2)$, symplectic form $\omega_X = d\theta \wedge dt_1 + ds \wedge dt_2$, and symplectic involution
	\[ I \colon X \to X \colon (\theta, s, t_1, t_2) \mapsto (\theta, -s, t_1, -t_2) \mathrm{.} \]
It contains Lagrangian submanifolds $L_0 = {S^1 \times (-\epsilon, \epsilon) \times \{ 0 \} \times \{ 0 \}}$ and $L_1 = S^1 \times  \{ 0 \} \times \{ 0 \} \times \bR$, which intersect each other cleanly in $\Gamma = S^1 \times \{ 0 \} \times \{ 0 \} \times \{ 0 \}$.
\begin{prop}
	\label{prop:the_local_model_is_a_local_model}
	There exists a symplectomorphism
	\[ \Psi \colon \cN(\Gamma) \rightarrow \cN(\gamma) \]
	from a neighborhood of $\Gamma$ in $X$ to a neighborhood of $\gamma \times \{ 0 \}$ in $\C^2$ such that
	\begin{enumerate}
		\item $\Psi(\Gamma) = \gamma \times \{ 0 \}$,
		\item $\Psi (L_0 \cap \cN(\Gamma))= L \cap \cN(\gamma)$,
		\item $R_\pi \circ \Psi = \Psi \circ I$, and
		\item $\Psi (L_1 \cap \cN(\Gamma)) = L_\phi \cap \cN(\gamma)$.
	\end{enumerate}
\end{prop}

\begin{proof}[Proof of Proposition \ref{prop:tori_can_be_smoothed_nicely}]
	Let $A = \{ (s,t_2) \in \bR^2 \colon s t_2 = 0 \}$ denote the union of the usual axes in Euclidean space.
	Under the map $\Psi^{-1}$ of Proposition \ref{prop:the_local_model_is_a_local_model}, the union of the Lagrangians $L  \cup L_\phi$ is modelled near $\gamma \times \{0\}$ as $S^1 \times A \times \{ 0 \}$, where we have exchanged coordinates $t_1$ and $t_2$.
	We pick a smoothing $B$ of $A \subset \bR^2$ near the origin whose components are exchanged by $I$ (which acts as rotation by $\pi$ on this plane).
	Observe that $S^1 \times B \times \{ 0 \}$ is Lagrangian with respect to $\omega_X$.
	Replacing $(L \cup L_\phi) \cap \cN(\gamma)$ by $\Psi((S^1 \times B \times \{0\}) \cap \cN(\Gamma))$ gives the desired smoothing.
\end{proof}

The technical work of this section, then, is to derive Proposition~\ref{prop:the_local_model_is_a_local_model}.  The next lemma is phrased for our situation and is a case of the \emph{Equivariant Darboux-Weinstein Theorem} \cite[Theorem 3.2]{hamil_group_action}.

\begin{lem}
	\label{lem:sketch_weinstein_result}
	Suppose that $\omega_0$ and $\omega_1$ are  symplectic forms in a neighborhood of $L_0 \subset X$ for which $L_0$ is Lagrangian and which satisfy $I^*(\omega_i) = \omega_i$ for $i = 0,1$.  Then there exist neighborhoods $U_0$ and $U_1$ of $\Gamma$ and a diffeomorphism
	\[ \sigma \colon U_0 \rightarrow U_1 \]
	such that $\sigma$ commutes with $I$, $\sigma^*\omega_1 = \omega_0$, and $\sigma$ restricts to the identity on $L_0 \cap U_0$. \qed
\end{lem}

Next we use Lemma~\ref{lem:sketch_weinstein_result} to obtain a local model for the Lagrangian $L$ near $\gamma \times \{ 0 \}$.
It establishes Proposition~\ref{prop:tori_can_be_smoothed_nicely} apart from the final item.

\begin{lem}
	\label{lem:good_tubular_neighborhood}
	There exists a symplectomorphism
	\[ F \colon \cN(\Gamma) \rightarrow \cN(\gamma) \]
	from a neighborhood $\cN(\Gamma)$ of $\Gamma$ in $X$ to a neighborhood $\cN(\gamma)$ of $\gamma \times \{ 0 \}$ in $\C^2$ such that

	\begin{enumerate}
		\item $F(\Gamma) = \gamma \times \{ 0 \}$,
		\item $F (L_0 \cap \cN(\Gamma))= L \cap \cN(\gamma)$, and
		\item $R_\pi \circ F = F \circ I$.
	\end{enumerate}
\end{lem}

\begin{proof}
	Parametrize the Jordan curve $\gamma \subset \C$ as $\gamma(\theta)$, where $\theta \in S^1$.  Now, $\gamma \times \{ 0 \}$ is a submanifold of $L$, so using the restriction of the standard metric on $\C^2$ to $L$, the exponential map
	\[ (\theta, s) \mapsto \exp_{(\gamma(\theta),0)} (s) \]
	identifies a neighborhood $S^1 \times (-\epsilon, \epsilon)$ of the normal bundle of $\gamma \times \{ 0 \}$ inside $L$ with a tubular neighborhood of $\gamma \times \{ 0 \}$ in $L$.  Since the standard metric is invariant under $R_\pi$, we have that $R_\pi$ preserves geodesics in $L$.  It follows that we have
	\[ (\theta, -s) \mapsto R_\pi(\exp_{(\gamma(\theta),0)}(s)) {\rm .} \]
	Next we take a smooth choice of orthonormal basis $\{ v^1_\theta, v^2_\theta \}$ for $(TL)^\perp\vert_{(\gamma(\theta),0)}$ (the orthogonal complement to $TL$ along $\gamma$) such that $v^1_\theta \in T_{(\gamma(\theta), 0 )}{\C \times \{ 0 \}}$.  Note that $(R_\pi)_*(v^1_\theta) = v^1_\theta$ and $(R_\pi)_*(v^2_\theta) = - v^2_\theta$.
	
	Decreasing $\epsilon$ and $\cN(\gamma)$ if necessary, we have a diffeomorphism
	\[ F' \colon S^1 \times (-\epsilon,\epsilon) \times (-\epsilon,\epsilon) \times (-\epsilon,\epsilon) \rightarrow \cN(\gamma) \colon
	(\theta, s, t_1, t_2) \mapsto
	\exp_{(\gamma(\theta),0)}(s) + t_1 v^1_\theta + t_2 v^2_\theta {\rm .} \]
	Since $R_\pi$ is a linear map, we have
	\[ F' \colon (\theta, -s, t_1, -t_2) \mapsto
	R_\pi(\exp_{(\gamma(\theta),0)}(s) + t_1 v^1_\theta + t_2 v^2_\theta) {\rm .} \]

Hence we observe that $F'$ satisfies all the required properties except possibly being a symplectomorphism. 
We now apply Lemma~\ref{lem:sketch_weinstein_result} to $\omega_X$ and to $F'^*(\omega)$.  Composing $F'$ with the resulting diffeomorphism $\sigma$ between neighborhoods of $\Gamma$ gives the required map $F$.
\end{proof}

It only remains to take account of the second Lagrangian $L_\phi$; we write $L_2 = F^{-1} (L_\phi)$.

\begin{lem}
	\label{lem:can_make_that_other_lagrangian_look_straight}
	There exists a symplectomorphism
	\[ G \colon U \rightarrow V \]
	defined on neighborhoods $U$ and $V$ of $\Gamma$ such that
	\begin{enumerate}
		\item $G$ restricts to the identity on $L_0$,
		\item $G \circ I = I \circ G$, and
		\item $G(L_2 \cap U)= L_1 \cap V$.
	\end{enumerate}
\end{lem}

\begin{proof}
	Observe that $L_2$ is a Lagrangian within a neighborhood of $\Gamma$ that satisfies $I(L_2) = L_2$ and that intersects $L_0$ cleanly in $\Gamma$.  
	Po\'zniak argues that there exist neighborhoods $U$ and $V$ of $\Gamma$ within which $L_2 \cap U$ is the graph of a function defined over $L_1$ \cite[Proposition 3.4.1 and Lemma 3.4.2]{pozniak}. Based on this feature he defines a map, which in our framework takes the form
	\[ G \colon U \rightarrow V \colon (\theta, s, t_1, t_2) \mapsto (\theta, s', t'_1, t_2), \]
	by the requirement that
	\[ (\theta, s - s', t_1 - t'_1, t_2) \in L_2 {\rm .} \]
	Such a map automatically satisfies properties (1) and (3), and Po\'zniak shows that $G$ is a symplectomorphism (this follows from $L_2$ being Lagrangian).
	To verify property (2) we directly compute
	\begin{align*}
		I \circ G (\theta, s, t_1, t_2) &= I(\theta, s', t'_1, t_2) = (\theta, -s', t'_1, -t_2) 
		= G(\theta, -s, t_1, -t_2) = G \circ I (\theta, s, t_1, t_2) {\rm ,}
	\end{align*}
where the third equality follows because
\[ (\theta, -(s - s'),t_1 - t'_1, -t_2) = I(\theta, s-s', t_1 - t'_1, t_2) \in I(L_2) = L_2 {\rm .} \qedhere \]
\end{proof}

\begin{proof}
[Proof of Proposition~\ref{prop:the_local_model_is_a_local_model}]
Set $\Psi = F \circ (G^{-1})$ with the maps $F$ and $G$ of Lemmas~\ref{lem:good_tubular_neighborhood} and \ref{lem:can_make_that_other_lagrangian_look_straight}.
\end{proof}

\section{Discussion.}
In 1911, Toeplitz posed the Square Peg Problem, which asks whether every continuous Jordan curve in the Euclidean plane {\em inscribes} (contains the vertices of) a square \cite{toeplitz1911}.
It remains open to this day.
The Rectangular Peg Problem (for smooth Jordan curves) grew out of it \cite[Conjecture 8]{matschke2014}.
Our solution fits into a long line of attack on these problems which involves identifying the inscribed feature with the (self-)intersection of an associated geometric-topological object.
The arguments tend to be quite short, once the appropriate outlook and auxiliary result is identified.

In 1913, Emch solved the Square Peg Problem for smooth convex curves \cite{emch1913}; and in 1929, Schnirelman solved it for smooth Jordan curves  \cite{schnirelman1929}.
In fact, both required weaker hypotheses than smoothness.
They laid the groundwork for later approaches, introducing the idea of configuration spaces and arguments involving homology and bordism.

In 1981, Vaughan gave a simple proof of the result that every continuous Jordan curve $\gamma$ inscribes a rectangle~\cite{meyerson1981}.  
Vaughan's argument was to define a continuous map $v : \Sym^2(\gamma) \to \C \times \bR_{\ge 0}$ by sending an unordered pair of points on $\gamma$ to the ordered pair consisting of their midpoint and the length of the line segment they span.
The points of self-intersection of $v$ thus parametrize inscribed rectangles in $\gamma$.
By filling $\gamma \times \{0\}$ with a disk $D \subset \C \times \{0\}$, we extend $v$ to a continuous map from $\bR {\mathbb P}^2$ to $\C \times \bR_{\ge 0} \subset \bR^3$ with the same set of self-intersections as $v$.
Such a map contains a point of self-intersection (in fact, a triple point), which corresponds to an inscribed rectangle in $\gamma$.
The fact that $v$ contains so much self-intersection indicates that a large family of inscribed rectangles should exist in $\gamma$, but extracting more information is a challenge.

In 1991, Griffiths claimed a solution of the smooth Rectangular Peg Problem based on elementary intersection theory, in the spirit of Schnirelman's work \cite{griffiths1991}.
However, in 2008, Matschke identified an irreparable error in its proof, casting doubt on the efficacy of this approach \cite{matschke2014}.
Following the discovery of this error, the status of the Rectangular Peg Problem reverted to the cases already reported.

In 2018, Hugelmeyer salvaged some new cases of the smooth Rectangular Peg Problem \cite{hugelmeyer2018}.
He did so by resolving Vaughan's map into a four-dimensional version that enables the detection of rectangles' \emph{aspect angles} (the angle between the two diagonals).
Define a map $h_n : \Sym^2(\gamma) \to \C \times \C$ by sending each unordered pair of points on $\gamma$ to their midpoint and the $(2n)$-th power of their difference.
For $n \ge 2$, the points of self-intersection of $h_n$ parametrize inscribed rectangles in $\gamma$ of aspect angle equal to an integer multiple of $\pi / n$.
Hugelmeyer showed how to identify $\im(h_n)$ with the image of a surface mapped into the $4$-ball with boundary on a $(2n,2n-1)$ torus knot in the $3$-sphere.
However, for $n \ge 3$, this knot does not bound a smoothly embedded M\"obius band in the $4$-ball: this is a result of Batson proven using Heegaard Floer homology \cite{batson}.
Hence $h_n$ contains a point of self-intersection for $n \ge 3$ when $\gamma$ is smooth.
In particular, taking $n=3$ leads to the novel case of the smooth Rectangular Peg Problem for a rectangle of aspect angle $\pi/3$.

In 2019, Hugelmeyer sharpened this approach and recovered $1/3$ of the smooth Rectangular Peg Problem \cite{hugelmeyer2019}.
More precisely, he showed that for any smooth Jordan curve $\gamma$, the set of values $\phi \in (0,\pi/2]$ for which $\gamma$ contains an inscribed rectangle of aspect angle $\phi$ has Lebesgue measure at least $\pi/6$.
The map $h_1$ above is a smooth embedding when $\gamma$ is a smooth Jordan curve, giving rise to a smooth M\"obius band $\im(h_1) = M \subset \C \times \C$.
The inscribed rectangles in $\gamma$ of aspect angle $\phi$ are parametrized by interior points of intersection between $M$ and $R_{2\phi}(M)$.
Hugelmeyer argued that this intersection is non-empty for $\ge 1/3$ of the angles $\phi$ by first introducing a novel ordering on a set of embedded M\"obius bands in $\C \times \bR_{\ge 0} \times S^1$ based on how they link and then applying a result from additive combinatorics.
In fact, this ordering may be applied to recover his earlier result, as well as the case of a square.

The inspiration behind our solution was to recast the problem within the framework of symplectic geometry, which offers greater rigidity for controlling intersections.
Following Hugelmeyer's second approach, we wished to endow $\C \times \C$ with a symplectic form with respect to which $M$ is Lagrangian and $R_{2\phi}$ is a Hamiltonian symplectomorphism.
Then an optimistic version of the Arnold-Givental conjecture predicts that $M$ and $M_\phi = R_{2\phi}(M)$ should contain at least $\dim H_*(M;\bZ/2\bZ) = 2$ points of intersection in their interiors.
Ultimately, we were able to arrange the framework by adjusting the map $h_1$ into the form $g \circ l$ given in the proof of the theorem.
We were able to circumvent proving the required version of the Arnold-Givental conjecture by noting that $M \cup M_\phi$ is a Lagrangian Klein bottle away from the common boundary of $M$ and $M_\phi$.
By smoothing it and appealing to Shevchishin's theorem, we obtained an intersection point that corresponds with the desired inscribed rectangle in $\gamma$ of aspect angle $\phi$.

Enjoyable accounts of the history of these problems and their relatives appear in \cite{kleewagon,matschke2014,pak}.
Additional notable progress appears in the work of Feller and Golla, Schwartz, and Tao \cite{fellergolla,schwartz,tao}.

\subsection*{Acknowledgements.}
We thank Peter Feller and Patrick Orson for stimulating discussions on a subtropical island at the outset of this work.
We thank Yasha Eliashberg, Joe Johns, and Leonid Polterovich for reassurances about Lagrangian smoothing, and Leonid in particular for steering us to the references \cite{makwu,pozniak}.

\bibliographystyle{amsplain}
\bibliography{references/works-cited.bib}
\end{document}